\documentclass[a4paper,12pt]{article}
\usepackage{amssymb}
\usepackage{amsmath,amsthm,amssymb}
\usepackage{amsfonts}

\newtheorem{tw}{Theorem}[section]
\newtheorem{lem}[tw]{Lemma}
\newtheorem{pro}[tw]{Proposition}

\theoremstyle{definition}

\newtheorem{exa}[tw]{Example}
\newtheorem{rem}[tw]{Remark}

\begin{document}

\begin{center}
{\Large Convex conjugates of analytic functions of logarithmically convex functionals}
\end{center}
\begin{center}
{\sc Krzysztof Zajkowski}\\
Institute of Mathematics, University of Bialystok \\ 
Akademicka 2, 15-267 Bialystok, Poland \\ 
E-mail:kryza@math.uwb.edu.pl 
\end{center}

\begin{abstract}
Let $f_{\bf c}(r)=\sum_{n=0}^\infty e^{c_n}r^n$ be an analytic function; ${\bf c}=(c_n)\in l_\infty$. We assume that $r$ is some logarithmically
convex and lower semicontinuous functional on a locally convex topological space $L$. In this paper we derive a formula on the Legendre-Fenchel transform of a functional
$\widehat{\lambda}({\bf c},\varphi)=\ln f_{\bf c}(e^{\lambda(\varphi)})$, where $\lambda(\varphi)=\ln r(\varphi)$ ($\varphi\in L$).
In this manner we generalize to the infinite case Theorem 3.1 from \cite{OZ1}. 

\end{abstract}



\section{Introduction}
In Convex analysis, it is natural to consider the Legendre-Fenchel transform of convex composite functions. The general rules of convex
conjugate calculus were obtained  in \cite{CLT1,CLT2}. Therein one can find the history of these investigations. The general formulas are
not always informative. It appears papers in which authors present forms of convex conjugates for some concrete compositions functions,
see for instance \cite{HiU,M-LS}. In this paper we would like to present investigations of the spectral radius of weighted composition operators
which lead to considerations of compositions with the so-called log-exponential function.

First we present a general result obtained
for the spectral radius of weighted composition operators.
Let $X$ be a  Hausdorff compact space with Borel measure $\mu$, $\alpha:X\mapsto X$ a continuous mapping preserving
$\mu$ (i.e. $\mu\circ\alpha^{-1}=\mu$) and $a$ be a continuous function on $X$.
Antonevich, Bakhtin and Lebedev constructed a functional $\tau_\alpha$, called $T$-entropy (see \cite{ABL2,ABL3}), on the set of probability
and $\alpha$-invariant measures
$M^1_\alpha$ such that for the spectral radius of the weighted composition operator $
(aT_\alpha)u(x)=a(x)u(\alpha(x))$ 
acting in $L^p$-spaces the following variational principle holds
\begin{equation}
\label{form2}
\ln r(aT_\alpha)=\max_{\nu\in M^1_\alpha}
\Big\{\int_X\ln|a|d\nu-\frac{\tau_\alpha(\nu)}{p}\Big\}.
\end{equation}
It turned out that $\tau_\alpha$ is nonnegative, convex and lower semicontinuous on $M^1_\alpha$.

For positive $a \in C(X)$ let $\varphi=\ln a$ and $\lambda(\varphi)=\ln r( e^{\varphi}T_\alpha)$. The functional $\lambda$ is convex and continuous
on $C(X)$ and 
the formula (\ref{form2}) states
that $\lambda$ is the Legendre-Fenchel transform of the function $\frac{\tau_\alpha}{p}$, i.e.
\begin{equation}
\label{vp}
\lambda(\varphi)=\max_{\nu\in M^1_\alpha}
\Big\{\int_X\varphi d\nu-\lambda^\ast(\nu)\Big\},
\end{equation}
where
$$
\lambda^{\ast}(\nu)=\left\{ \begin{array}{ll}
\frac{\tau_\alpha(\nu)}{p}, \qquad \nu\in M_\alpha^1 \\[8pt]   
+\infty, \qquad  {\rm otherwise .}  
\end{array} \right.\\
$$
It means that the effective domain $D(\lambda^{\ast})$ is contained in $M^1_{\alpha}$. 

In operators algebras, it is natural to consider functions of operators. In the context of weighted composition operators, a problem arises 
how  the formula (\ref{vp}) changes when instead of $aT_\alpha$ we take  some functions of one.

First results was obtained for polynomials of $aT_\alpha$. Let $\sum_{n=0}^Na_nz^n$ be a polynomial with positive coefficients $a_n$.
The convex conjugate of $\widetilde{\lambda}(\varphi)=\ln r(\sum_{n=0}^Na_n(e^\varphi T_\alpha)^n)$ is equal to
\begin{equation}
\label{lambdaf}
\widetilde{\lambda}^{\ast}(m)= m(X) \lambda^\ast\Big(\frac{m}{m(X)}\Big) + \min_{{\bf t} \in S_{m(X)}} \sum_{n=0}^N t_n \ln \frac{t_n}{a_n}, 
\end{equation}
where $S_{m(X)}=\{(t_n)_{n=0}^N :\;t_n\ge 0,\; \sum_{n=0}^N t_n=1\;{\rm and}\; \sum_{n=0}^N n t_n = m(X)\}$ and the effective domain of $\widetilde{\lambda}^{\ast}$ is contained
in the set 
$
\{m \in C(X)^{\ast}: m \in M_{\alpha}\;\; {\rm and}\;\;   m(X) \in [0,N] \};
$
$M_{\alpha}$ is the set of all $\alpha$-invariant measures on $X$ (see \cite{OZ1,OZ2} for more details) .

Let $c_n$ denote $\ln a_n$. If we consider dependence of the logarithm of the spectral radius also on the vector $(c_n)_{n=0}^N$, that is
if we consider the functional $\widehat{\lambda}((c_n),\varphi)=\ln r(\sum_{n=0}^Ne^{c_n}r(e^\varphi T_\alpha)^n)$, then the convex conjugate
$\widehat{\lambda}$ has the form
\begin{equation}\label{Ltran} 
\widehat{\lambda}^{\ast}((t_n),\bar{\mu})=\bar{\mu}(X) \lambda^\ast\Big(\frac{\bar{\mu}}{\bar{\mu}(X)}\Big) + \sum_{n=0}^N t_n \ln t_n 
\end{equation}
and the effective domain of $\widehat{\lambda}^{\ast}$ is contained in the set 
$$
\mathcal{M}=\Big\{((t_n),\bar{\mu}):\;t_n\ge 0,\; \sum_{n=0}^N t_n=1,\; \bar{\mu}\in M_\alpha\;{\rm and\; }\; \bar{\mu}(X) = \sum_{n=0}^N n t_n \Big\}
$$ 
(see \cite{OZ1} for more details).

The generalization of the formula (\ref{lambdaf}) on the case of analytic functions was obtained in \cite{OZ3} where authors proved that for
$\widetilde{\lambda}(\varphi)=\ln r(\sum_{n=0}^\infty a_n(e^\varphi T_\alpha)^n)$ (we present this result only in the case of all $a_n>0$) the convex conjugate has the form
$$
\widetilde{\lambda}^{\ast}(m)=m(X) \lambda^\ast\Big(\frac{m}{m(X)}\Big) + \min_{(t_n) \in S_{m(X)}} \liminf_{N\to \infty} \sum_{n=0}^N t_n \ln \frac{t_n}{a_n},  
$$
where $S_{m(X)}=\{(t_n)_{n=0}^\infty:\;t_n\ge 0,\; \sum_{n=0}^\infty t_n=1,\;\sum_{n=0}^\infty n t_n<+\infty\;{\rm and}\; \sum_{n=0}^\infty nt_n=m(X)\}$ and the effective domain
of $\widetilde{\lambda}^{\ast}$ is contained in the set of all $\alpha$-invariant measures $M_\alpha$.

To generalize the formula (\ref{Ltran}) we had to solve two problems: to define the functional $\widehat{\lambda}$ for infinite number
of variables $c_n$ and to consider the entropy function with infinite numbers of summands. It is known that not for all infinite
sequences of probability weights $(t_n)$ the entropy function $\sum_{n=0}^\infty t_n\ln t_n$ takes finite values (see Example \ref{przyk}). A solution of these problems there is in Theorem \ref{gen}.

\section{Convex conjugates}
We begin by recalling Proposition 2.3 from \cite{OZ3}.
\begin{pro} 
\label{logseries}
For the sequence $(b_n)_{n=0}^{\infty}$ such that $b_n> 0$  and  $\sum_{n=0}^{\infty}b_n < \infty$ the following holds 
$$ 
\ln\sum_{n=0}^\infty b_n  = \max_{\substack{t_n\ge 0,\; \sum t_n=1}} 
\limsup_{N\to\infty}\sum_{n=0}^N (t_n \ln b_n -t_n\ln t_n);
$$ 
it is  taken that $0\ln 0=0$.
This maximum is attained at $t_n=\frac{b_n}{\sum_{n=0}^\infty b_n}$.
\end{pro}
 

\begin{exa}
Let $b_n$ equals $r^n$ for $r\in (0,1)$. Then the sum of the series $\sum_{n=0}^\infty r^n=1/(1-r)$ and by the above Proposition we obtain formula
\begin{equation}
\label{geom}
-\ln(1-r)=\max_{\substack{t_n\ge 0,\; \sum t_n=1}} \limsup_{N\to\infty}\sum_{n=0}^N (nt_n \ln r -t_n\ln t_n)
\end{equation}
that can be rewritten as follows
\begin{equation}
\label{geom1}
\ln(1-r)=\min_{\substack{t_n\ge 0,\; \sum t_n=1}} \liminf_{N\to\infty}\sum_{n=0}^N t_n\ln\frac{t_n}{r^n}. 
\end{equation}

By Proposition \ref{logseries}, it is known that in this case the above maximum is attained for the geometric distribution, i.e. at $t_n=(1-r)r^n$.
Let us emphasize that for the geometric distribution $((1-r)r^n)_{n=0}^\infty$ the series 
$\sum_{n=0}^\infty nt_n=\sum_{n=0}^\infty n(1-r)r^n=r/(1-r)$ is convergent.
For this reason  we can search  for the  maximum in (\ref{geom})
under the additional restricted condition $\sum_{n=0}^\infty nt_n<\infty $. So
we can rewrite (\ref{geom}) in the form 
$$
-\ln(1-r)=\max_{\substack{t_n\ge 0,\; \sum t_n=1 \\ \sum nt_n<\infty }}\Big\{\ln r\sum_{n=0}^\infty nt_n  -\sum_{n=0}^\infty t_n\ln t_n\Big\},
$$
where the series $\sum_{n=0}^\infty t_n\ln t_n$ is either convergent or divergent to minus infinity. But the second opportunity is not possible because then the above maximum  will be equal to $+\infty$ and it will not  to be finite. It means that the restricted condition $\sum_{n=0}^\infty nt_n<+\infty$ ensures the convergence of the series $\sum_{n=0}^\infty t_n\ln t_n$. 
\end{exa}
Basing on this example we can formulate the following
\begin{pro}
\label{entf}
If a probability distribution $(t_n)$ satisfies the condition $\sum_{n=0}^\infty nt_n<\infty$ then the series $\sum_{n=0}^\infty t_n\ln t_n$ is convergent.
\end{pro}
The above statement is only a sufficient condition for the convergence of the series $\sum_{n=0}^\infty t_n\ln t_n$.
\begin{exa}
Using for instance the sum of series $\sum_{n=1}^\infty 1/n^2=\pi^2/6$ and taking $t_n=6/(\pi n)^2$ we see that the series 
$\sum_{n=1}^\infty nt_n=\sum_{n=1}^\infty 6/(\pi^2n)$
is divergent but the series 
$$
\sum_{n=1}^\infty t_n\ln t_n= \frac{6}{\pi^2}\sum_{n=1}^\infty \frac{1}{n^2}\Big(\ln\frac{6}{\pi^2}-2\ln n\Big)=
\ln\frac{6}{\pi^2}-\frac{12}{\pi^2}\sum_{n=1}^\infty\frac{\ln n}{n^2}
$$ 
is convergent.
\end{exa}

Obviously   not for all probability distributions $(t_n)$ series of the form $\sum_{n=1}^\infty t_n\ln t_n$ are convergent. 
\begin{exa}
\label{przyk}
Using Cauchy condensation test one can check that the series $\sum_{n=2}^\infty 1/(n(\ln n)^2)$ is convergent and the series 
$\sum_{n=2}^\infty 1/(n\ln n)$ divergent. Denoting the sum $\sum_{n=2}^\infty 1/(n(\ln n)^2)$ by $a$ and taking $t_n= 1/(n(\ln n)^2a)$
one can check that 
\begin{eqnarray*}
\sum_{n=2}^\infty t_n\ln t_n & = & \sum_{n=2}^\infty \frac{1}{n(\ln n)^2a}\ln\frac{1}{n(\ln n)^2a}\\
\; & = & -\ln a-\frac{1}{a}\sum_{n=2}^\infty \frac{1}{n\ln n}
-\frac{2}{a}\sum_{n=2}^\infty \frac{\ln\ln n}{n(\ln n)^2}=-\infty.
\end{eqnarray*}
\end{exa}

For each $(c_n)\in l_\infty$  the radius of convergence of an analytic function $f_{\bf c}(r)=\sum_{n=0}^\infty e^{c_n}r^n$  equals 1. Taking in Proposition \ref{logseries} $b_n=e^{c_n}r^n$ for $r\in(0,1)$ and $(c_n)\in l_\infty$ we obtain
\begin{eqnarray}
\label{angeom}
\ln\sum_{n=0}^\infty e^{c_n}r^n & = & \max_{\substack{t_n\ge 0,\; \sum t_n=1}} \limsup_{N\to\infty}\sum_{n=0}^N (c_nt_n+nt_n \ln r -t_n\ln t_n)\nonumber\\
\; & = & \max_{\substack{t_n\ge 0,\; \sum t_n=1}}\Big\{\sum_{n=0}^\infty c_nt_n-\liminf_{N\to\infty}\sum_{n=0}^N t_n\ln\frac{t_n}{r^n}\Big\}.
\end{eqnarray}
For each $(c_n)\in l_\infty$ the series $\sum_{n=0}^\infty c_nt_n$ is the  bounded linear functional on the space $l_1$.
Let $S$ denote the infinite standard symplex $\{{\bf t}=(t_n)_{n=0}^\infty:\;t_n\ge 0,\; \sum t_n=1\}$. 
Notice that $S\subset l_1$ and $l_\infty\simeq (l_1)^\ast$.

Consider now the expression $\ln\sum_{n=0}^\infty e^{c_n}r^n$ as the functional $\lambda_r:l_\infty\mapsto \mathbb{R}$, i.e.
$$
\lambda_r({\bf c})=\ln\sum_{n=0}^\infty e^{c_n}r^n \quad {\rm for }\quad {\bf c}=(c_n)\in l_\infty\quad(r\in(0,1)).  
$$
The expression (\ref{angeom}) means that $\lambda_r$ is the convex conjugate of a functional $h_r:l_1\mapsto \mathbb{R}$
defined as follows
$$
h_r({\bf t})=\left\{ \begin{array}{lll}
\liminf_{N\to\infty}\sum_{n=0}^N t_n\ln\frac{t_n}{r^n}\quad &{\rm if}&\quad {\bf t}\in S,\\   
+\infty \qquad & {\rm if} &\quad {\bf t}\in l_1\setminus S .  
\end{array} \right.\\
$$
The effective domain of $h_r$ is contained in $S$. Moreover by (\ref{geom1}) we have that $h_r({\bf t})\ge\ln(1-r)$ for ${\bf t}\in l_1$.
Since the function $x\ln(x/a)$ $(a>0)$ is convex on $(0,+\infty)$ we have that $\sum_{n=0}^N t_n\ln\frac{t_n}{r^n}$ is convex
on $(0,+\infty)^{N+1}$ but because in the definition of $h_r$ appears the lower limit then one (we) can not prove convexity of it on $S$.
For these reasons we only get that $\lambda_r$ is the convex conjugate of $h_r$  and $\lambda_r^\ast$ is convex  and lower semicontinuous regularization of $h_r$.

By $\widetilde{S}$ we will denote the subset of sequences belonging to $S$ and satisfying the condition $\sum nt_n<\infty$. 
By Proposition \ref{logseries} the maximum in (\ref{angeom}) is attained at the sequence $(e^{c_n}r^n/f_{\bf c}(r))$. Notice that 
for this sequence the value
$$
\sum_{n=0}^\infty nt_n=\frac{r}{f_{\bf c}(r)}\sum_{n=1}^\infty ne^{c_n}r^{n-1}=\frac{rf_{\bf c}'(r)}{f_{\bf c}(r)} 
$$
is finite for $r\in(0,1)$.
Thus, using Proposition \ref{entf}, we can express (\ref{angeom}) as follows
\begin{equation}
\label{angeom1}
\ln\sum_{n=0}^\infty e^{c_n}r^n=\max_{\substack{{\bf t}\in\widetilde{S}}}\Big\{\sum_{n=0}^\infty c_nt_n-\Big(\sum_{n=0}^\infty t_n\ln t_n-\ln r\sum_{n=0}^\infty nt_n \Big)\Big\}.
\end{equation}
\begin{rem}
Considering the logarithm of the finite sum $\sum_{n=0}^N e^{c_n}r^n$ we can assume that $r$ is any nonnegative number. All series in the above
formula become finite sums, $\widetilde{S}$ will be the $(N+1)$-dimensional standard symplex and for $r=1$ we get the classical variational principle
for the so-called log-exponential function (see Example 11.12 in \cite{RocWet}).

\end{rem}
Define now the function $g_r$ on $l_1$ in the following way
$$
g_r({\bf t})=\left\{ \begin{array}{lll}
\sum_{n=0}^\infty t_n\ln t_n - \ln r\sum_{n=0}^\infty nt_n \qquad & {\rm if}& {\bf t}\in\widetilde{S},\\   
+\infty \qquad  & {\rm if}& {\bf t}\in l_1\setminus \widetilde{S} .  
\end{array} \right.\\
$$
The series $\sum_{n=0}^\infty nt_n$ is a linear but unbounded functional on $l_1$. Because the space $l_0$ of sequences with finite supports
is dense in $l_1$ and $\sum_{n=0}^\infty nt_n<\infty$ for ${\bf t}\in l_0$ then the set $\widetilde{S}$ 
is dense in $S$; besides it is a convex subset of $S$. 
The function $g_r$ is convex on
$l_1$ and we can defined $\lambda_r^\ast$  as the lower semicontinuous regularization of $g_r$.

Let $\lambda$ denote $\ln r$. If $r\in(0,1)$ then $\lambda\in(-\infty, 0)$. 
Consider now the expression $\ln\sum_{n=0}^\infty e^{c_n}r^n=\ln\sum_{n=0}^\infty e^{c_n+n\lambda}$ as a functional 
$\tilde{\lambda}:l_\infty\times \mathbb{R}\mapsto\mathbb{R}\cup\{+\infty\}$ defined
in the following way

\begin{equation}
\label{lambda1}
\tilde{\lambda}({\bf c},\lambda)=\left\{ \begin{array}{lll}
\ln\sum_{n=0}^\infty e^{c_n+n\lambda} \quad & {\rm if}& ({\bf c},\lambda)\in l_\infty\times (-\infty,0),\\   
+\infty  \qquad  {\rm otherwise .}  
\end{array} \right.\\
\end{equation}
We can rewrite (\ref{angeom1}) as
\begin{equation}
\label{angeom3}
\tilde{\lambda}({\bf c},\lambda)=\max_{\substack{{\bf t}\in\widetilde{S}}}\Big\{\sum_{n=0}^\infty c_nt_n+\lambda\sum_{n=0}^\infty nt_n -\sum_{n=0}^\infty t_n\ln t_n\Big\}.
\end{equation}
By fixing ${\bf t}\in\widetilde{S}$ the expression $\lambda\sum_{n=0}^\infty nt_n -\sum_{n=0}^\infty t_n\ln t_n$, as a linear function of the variable $\lambda$, is the convex
conjugate of a function
\begin{equation}
\label{angeom2}
f_{\bf t}(a)=\left\{ \begin{array}{lll}
\sum_{n=0}^\infty t_n\ln t_n \qquad {\rm if}\quad a=\sum_{n=0}^\infty nt_n,\\
+\infty \qquad  {\rm otherwise},  
\end{array} \right.\\
\end{equation}
that is   $\lambda\sum_{n=0}^\infty nt_n -\sum_{n=0}^\infty t_n\ln t_n=\max_{a\in\mathbb{R}}\{a\lambda-f_{\bf t}(a)\}$. Since $a=\sum_{n=0}^\infty nt_n$
is always nonnegative, we may search for this maximum only over the set $\mathbb{R}_+=[0,+\infty)$.
 
Changing ${\bf t}$ and taking $f({\bf t},a)=f_{\bf t}(a)$ we can express $\tilde{\lambda}$ as follows
\begin{eqnarray}
\label{form1}
\tilde{\lambda}({\bf c},\lambda) & = & \max_{\substack{{\bf t}\in\widetilde{S} }}\Big\{\sum_{n=0}^\infty c_nt_n+
\max_{a\in\mathbb{R}_+}\{a\lambda-f({\bf t},a)\}\Big\}\nonumber\\
\; & = & \max_{{\bf t}\in\widetilde{S} }\max_{a\in\mathbb{R}_+}\Big\{\sum_{n=0}^\infty c_nt_n+\lambda a-f({\bf t},a)\Big\}.
\end{eqnarray}
Define now the set 
$$
D_{\tilde{\tau}}=\{({\bf t},a)\in l_1\times \mathbb{R}:\;{\bf t}\in\widetilde{S},\;a\ge 0\;{\rm and}\;a=\sum_{n=0}^\infty nt_n\}
$$
and the function
\begin{equation}
\label{tau}
\tilde{\tau}({\bf t},a)=\left\{ \begin{array}{lll}
\sum_{n=0}^\infty t_n\ln t_n \quad & {\rm if}&({\bf t},a)\in D_{\tilde{\tau}},\\
+\infty \qquad  {\rm otherwise}.  
\end{array} \right.
\end{equation}
Now we can rewrite (\ref{form1}) as follows
$$
\tilde{\lambda}({\bf c},\lambda)=\max_{({\bf t},a)\in D_{\tilde{\tau}}}\Big\{\sum_{n=0}^\infty c_nt_n+\lambda a-\sum_{n=0}^\infty t_n\ln t_n\Big\}.
$$
Notice that for $({\bf c},a)\in l_\infty\times\mathbb{R}$ the expression $\sum_{n=0}^\infty c_nt_n+\lambda a$ is a linear and bounded functional on the space $l_1\times \mathbb{R}$  and the above formula
means that $\tilde{\lambda}$ is the convex conjugate of $\tilde{\tau}$. 


\begin{pro}
The convex conjugate of $\tilde{\lambda}$ defined by (\ref{lambda1}) is the lower semicontinuous regularization of the functional $\tilde{\tau}$ defined by (\ref{tau}).
\end{pro}
\begin{proof}
We should prove that $D_{\tilde{\tau}}$ is a convex subset of $l_1\times \mathbb{R}$ and $\tilde{\tau}$ is a convex function on it.
Let $({\bf t}^1,a_1)$, $({\bf t}^2,a_2)$ belong to $D_{\tilde{\tau}}$ and $s$ in the interval $(0,1)$. Consider the element
$$
s({\bf t}^1,a_1)+(1-s)({\bf t}^2,a_2)=(s{\bf t}^1+(1-s){\bf t}^2,sa_1+(1-s)a_2).
$$
Since $\widetilde{S}$ is convex, $s{\bf t}^1+(1-s){\bf t}^2\in \widetilde{S}$ for ${\bf t}^1,\;{\bf t}^2\in \widetilde{S}$ and $s\in (0,1)$.
Moreover, if $\sum_{n=0}^\infty nt_n^1=a_1$ and $\sum_{n=0}^\infty nt_n^2=a_2$ then $sa_1+(1-s)a_2=\sum_{n=0}^\infty n(st_n^1+(1-s)t_n^2)$.
It follows the convexity of $D_{\tilde{\tau}}$ and the convexity of $\tilde{\tau}$ results from the convexity of the entropy function.

\end{proof}

Let now $r:L\mapsto (0,+\infty]$ be a logarithmically convex and lower semicontinuous functional defined on a locally convex topological space $L$
and  $\lambda(\varphi)=\ln r(\varphi)$ for $\varphi\in L$. The functional $\lambda$ is convex and lower semicontinuous.
Let $D(\lambda^\ast)$ denote the effective domain of the convex conjugate of $\lambda$  . For the functional $\lambda:L\mapsto \mathbb{R}\cup\{+\infty\}$ the following 
variational principle holds
\begin{equation}
\label{logconv}
\lambda(\varphi)=\sup_{\mu\in D(l^\ast)}\{\left\langle\mu,\varphi\right\rangle-\lambda^\ast(\mu)\},\qquad\varphi\in L. 
\end{equation}
Let $\mathcal{D}_\lambda$ denote a set $\{\varphi\in L:\;\lambda(\varphi)<0\}$. We assume that $\mathcal{D}_\lambda$ is nonempty set. 
Define now the functional $\widehat{\lambda}:l_\infty\times L\mapsto \mathbb{R}\cup\{+\infty\}$ as follows 
\begin{equation}
\label{hatlam}
\widehat{\lambda}({\bf c},\varphi)=\left\{ \begin{array}{lll}
\ln\sum_{n=0}^\infty e^{c_n+n\lambda(\varphi)} \quad & {\rm if}&\quad({\bf c},\varphi)\in l_\infty\times \mathcal{D}_\lambda,\\
+\infty \qquad  {\rm otherwise}.  
\end{array} \right.
\end{equation}

The effective domain $D(\widehat{\lambda})$ is included in $l_\infty\times \mathcal{D}_\lambda$.  
Substituting (\ref{logconv}) into (\ref{angeom3}) we obtain
\begin{eqnarray}
\label{angeom4}
\widehat{\lambda}({\bf c},\varphi)& = & \max_{\substack{{\bf t}\in\widetilde{S} }}
\Big\{\sum_{n=0}^\infty c_nt_n+\sup_{\mu\in D(s^\ast)}\{\left\langle\mu,\varphi\right\rangle-\lambda^\ast(\mu)\}\sum_{n=0}^\infty nt_n -\sum_{n=0}^\infty t_n\ln t_n\Big\}\nonumber\\
\; & = & \max_{\substack{{\bf t}\in\widetilde{S}}}\sup_{\mu\in D(\lambda^\ast)}
\Big\{\sum_{n=0}^\infty c_nt_n+\left\langle\Big(\sum_{n=0}^\infty nt_n\Big)\mu,\varphi\right\rangle-\Big(\sum_{n=0}^\infty nt_n\Big)\lambda^\ast(\mu)-\sum_{n=0}^\infty t_n\ln t_n\Big\}.\nonumber\\
\end{eqnarray}
Notice that for ${\bf t}\in\widetilde{S}$ the expression $\sum_{n=0}^\infty nt_n$ can be any nonnegative number, moreover 
$\sum_{n=0}^\infty nt_n=0$ if and only if ${\bf t}=e_0$ ($e_0=(1,0,0,...)$). The value of the above expression in the curly brackets at $e_0$ is equal to $c_0$ 
and is less than $\widehat{\lambda}({\bf c},\varphi)=\ln\sum_{n=0}^\infty e^{c_n+n\lambda(\varphi)}$ for any ${\bf c}\in l_\infty$ and 
$\varphi\in \mathcal{D}_\lambda$. For this reason we can search for the above maximum without the point $e_0$.

Let $\bar{\mu}$ denote  $\Big(\sum_{n=0}^\infty nt_n\Big)\mu$ and 

$$
D_{\widehat{\tau}}=\Big\{({\bf t},\bar{\mu})\in l_1\times L^\ast:\;{\bf t}\in \widetilde{S}\setminus \{e_0\}\;{\rm and}\;\frac{\bar{\mu}}{\sum_{n=0}^\infty nt_n}\in D(\lambda^\ast)\Big\}.
$$
Define the functional $\widehat{\tau}:l_1\times L^\ast\mapsto \mathbb{R}\cup \{+\infty\}$ in the following way
\begin{equation}
\label{hattau}
\widehat{\tau}({\bf t},\bar{\mu})=\left\{ \begin{array}{ll}
(\sum_{n=0}^\infty nt_n)\lambda^\ast\Big(\frac{\bar{\mu}}{\sum_{n=0}^\infty nt_n}\Big)+\sum_{n=0}^\infty t_n\ln t_n \quad {\rm if}\quad ({\bf t},\bar{\mu})\in D_{\widehat{\tau}},\\
+\infty  \qquad  {\rm otherwise}.  
\end{array} \right.\\
\end{equation}
Now we can rewrite (\ref{angeom4}) as follows
$$
\widehat{\lambda}({\bf c},\varphi)=\sup_{({\bf t},\bar{\mu})\in D_{\widehat{\tau}}}\Big\{\sum_{n=0}^\infty c_nt_n+\left\langle\bar{\mu},\varphi\right\rangle-\widehat{\tau}({\bf t},\bar{\mu})\Big\}.
$$
It means that $\widehat{\lambda}$ is the convex conjugate of $\widehat{\tau}$ that is $\widehat{\lambda}=\widehat{\tau}^\ast$.
\begin{pro}
\label{formhattau}
The convex conjugate of $\widehat{\lambda}$ defined by (\ref{hatlam}) is the lower semicontinuous regularization of the functional $\widehat{\tau}$ defined by (\ref{hattau}).
\end{pro}
\begin{proof}
Let $({\bf t}_1,\bar{\mu}_1)$ and $({\bf t}_2,\bar{\mu}_2)$ belong to $D_{\widehat{\tau}}$. Consider an element 
$$
s({\bf t}_1,\bar{\mu}_1)+(1-s)({\bf t}_2,\bar{\mu}_2)=(s{\bf t}_1+(1-s){\bf t}_2,s\bar{\mu}_1+(1-s)\bar{\mu}_2)
$$
for $s\in(0,1)$. Since $\widetilde{S}\setminus \{e_0\}$ is convex, $s{\bf t}_1+(1-s){\bf t}_2$ in $\widetilde{S}\setminus \{e_0\}$. Let
$a_1$  and $a_2$ denote $\sum_{n=0}^\infty nt_n^1$ and $\sum_{n=0}^\infty nt_n^2$, respectively. We should check that
\begin{equation}
\label{nal}
\frac{s\bar{\mu}_1+(1-s)\bar{\mu}_2}{sa_1+(1-s)a_2}\in D(\lambda^\ast).
\end{equation}
Notice that $\bar{\mu}_1/a_1$, $\bar{\mu}_2/a_2$ in $D(\lambda^\ast)$ and 
$$
\frac{s\bar{\mu}_1+(1-s)\bar{\mu}_2}{sa_1+(1-s)a_2}=\frac{sa_1}{sa_1+(1-s)a_2}\frac{\bar{\mu}_1}{a_1}+
\frac{(1-s)a_2}{sa_1+(1-s)a_2}\frac{\bar{\mu}_2}{a_2}.
$$
Because $D(\lambda^\ast)$ is convex and $sa_1/(sa_1+(1-s)a_2)$, $[(1-s)a_2]/(sa_1+(1-s)a_2)$ are positive and their sum to be $1$
then (\ref{nal}) is valid.

The entropy function $\sum_{n=0}^\infty t_n\ln t_n$ is convex on $\widetilde{S}$ and the convexity of $\lambda^\ast$ implies
\begin{equation*}
(sa_1+(1-s)a_2)\lambda^\ast\Big(\frac{s\bar{\mu}_1+(1-s)\bar{\mu}_2}{sa_1+(1-s)a_2}\Big) 
 \le sa_1\lambda^\ast\Big(\frac{\bar{\mu}_1}{a_1}\Big)+(1-s)a_2\lambda^\ast\Big(\frac{\bar{\mu}_2}{a_2}\Big).
\end{equation*}
For these reasons the functional $\widehat{\tau}$ is convex on $l_1$ which completes the proof.

\end{proof}

Return now to the spectral radius of the weighted composition operator $e^\varphi T_\alpha$. This operator, considered in $L^p$-spaces
(Banach lattices), is an example of positive operators. The spectral radii of such operators belong to their spectrums (see Prop. 4.1 in Ch. V in \cite{Sch}).
Analytic functions with positive coefficients of positive operators are also positive ones and their spectral radii also belong
to their spectrums. This fact yields that $r(f_{\bf c}(e^\varphi T_\alpha))= f_{\bf c}(r(e^\varphi T_\alpha))$. This means that investigating
the spectral radius $r(f_{\bf c}(e^\varphi T_\alpha))$ we can consider the functions of the spectral radius $f_{\bf c}(r(e^\varphi T_\alpha))$.
Since $r(e^\varphi T_\alpha)$ depends logarithmically on $\varphi$, we have investigated $\ln f(e^{\lambda(\varphi)})$, where 
$\lambda(\varphi)= \ln r(e^\varphi T_\alpha)$. 

Proposition \ref{formhattau} gives us the general form of the Legendre-Fenchel transform of this functional. In our case we can precise
some details. Because $D(\lambda^\ast)$ is included in the set of all probability and $\alpha$-invariant measures $M_\alpha^1$ then
$\bar{\mu}=(\sum_{n=0}^\infty nt_n)\mu$ (${\bf t}\in\widetilde{S}$ and $\mu\in M_\alpha^1$) is any $\alpha$-invariant measure. For
$\bar{\mu}\equiv 0$, using the Legendre-Fenchel transform, in the same way as in the proof of the Theorem 3.1 form \cite{OZ1}, one can
calculate that $\lambda^\ast(e_0,{\bf 0})=0$. In this way we can formulate a generalization of  Theorem 3.1 form \cite{OZ1}. 
\begin{tw}
\label{gen}
Let $X$ be a  Hausdorff compact space with Borel measure $\mu$, $\alpha:X\mapsto X$ a continuous mapping preserving
$\mu$  and $\varphi$ be a continuous function on $X$. Let $\lambda:C(X)\mapsto \mathbb{R}$ denote a functional being the spectral exponent of  weighted composition
operators $e^\varphi T_\alpha$ acting in the space $L^p(X,\mu)$; $\lambda(\varphi)=\ln r(e^\varphi T_\alpha)$. Then the Legendre-Fenchel transform of  $\widehat{\lambda}$ defined 
by (\ref{hatlam}) is the lower semicontinuous regularization of a functional  defined by the following formula
$$
\widehat{\lambda}^\ast({\bf t},\bar{\mu})=\frac{1}{p}\Big(\sum_{n=0}^\infty nt_n\Big)\tau_\alpha\Big(\frac{\bar{\mu}}{\sum_{n=0}^\infty nt_n}\Big)+\sum_{n=0}^\infty t_n\ln t_n 
$$
on the set $\{({\bf t},\bar{\mu})\in l_1\times C(X)^\ast:\;{\bf t}\in \widetilde{S}\setminus \{e_0\},\;\bar{\mu}\in M_\alpha\; {\rm and}\;
\bar{\mu}(X)=\sum_{n=0}^\infty nt_n\}$ and $+\infty$ otherwise. The functional $\tau_\alpha$ is $T$-entropy. When $\bar{\mu}(X)=0$ then $\widehat{\lambda}^\ast$ takes 
value zero.
\end{tw}

\end{document}